\documentclass[a4paper,11pt]{article}
\usepackage{latexsym}
\usepackage{amsmath}
\usepackage{amsthm}
\usepackage[pdftex]{graphicx}
\usepackage{float}
\usepackage{mathrsfs}
\usepackage{hyperref}									% Hyper-linking.
\usepackage{amssymb}
\usepackage{amsfonts}

\theoremstyle{thm} \newtheorem{thm}{Theorem}

\newtheorem{lem}[thm]{Lemma}
\newtheorem{cor}[thm]{Corollary}
\theoremstyle{definition} \newtheorem{mydef}[thm]{Definition}
\theoremstyle{remark} \newtheorem{rem}[thm]{Remark}
\theoremstyle{definition} 

\newcommand{\C}{\mathbb C}								% set of complex numbers
\newcommand{\Z}{\mathbb Z}								% set of integers
\newcommand{\tth}{\text{th}}								% nth
\newcommand{\imod}[1]{\left( \bmod{#1} \right)}					% (mod)

\begin{document}

%--------------------------------------------------------------------------------------------------------------------------------------------------------
\title{Complete Residue Systems: A Primer and an Application}
\author{Pietro Paparella\thanks{Department of Mathematics, Washington State University, Pullman, WA 99164-3113, USA (\href{mailto: ppaparella@math.wsu.edu}{\texttt{ppaparella@math.wsu.edu}}). The material in this paper is part of the author's doctoral dissertation in preparation at Washington State University.}}
\date{}
\maketitle
%--------------------------------------------------------------------------------------------------------------------------------------------------------

%--------------------------------------------------------------------------------------------------------------------------------------------------------
\begin{abstract}
\noindent
{\it Complete residue systems} play an integral role in abstract algebra and number theory, and a description is typically found in any number theory textbook. This note provides a concise overview of complete residue systems, including a robust definition, several well-known results, a proof to the converse of a well-known theorem, ancillary results pertaining to an application arising from the study of the roots of nonnegative matrices, and extends our knowledge of complete residue systems in relation to complete sets of roots of unity.

\end{abstract}
%--------------------------------------------------------------------------------------------------------------------------------------------------------

%--------------------------------------------------------------------------------------------------------------------------------------------------------
\section{Introduction}
%--------------------------------------------------------------------------------------------------------------------------------------------------------

Complete residue systems play an integral role in abstract algebra and number theory, and a description is typically found in any number theory textbook (see, for instance, \cite{MR1382656,MR2445243,MR0092794}). In Section 1, we provide a concise overview of complete residue systems, including a robust definition, several well-known results, and a proof to the converse of a well-known theorem. In Section 2, we derive ancillary results pertaining to an application arising from the study of the roots of nonnegative matrices, which extend, to a certain extent, Theorems 65 and 66 in \cite{MR2445243}.

%--------------------------------------------------------------------------------------------------------------------------------------------------------
\section{Complete Residue Systems}
%--------------------------------------------------------------------------------------------------------------------------------------------------------

Our discussion of complete residue systems begins with the following definition.

% Defn 1-----------------------------------------------------------------------------------------------------------------------------------------------
\begin{mydef}
Let $R = \left\{ a_0, a_1, \dots, a_{h-1} \right\}$ be a set of integers, and for $h>1$ an integer, let $R(h) = \{ 0, 1, \dots, h-1\}$. Then $R$ is said to be a {\it complete residue system} $\imod{h}$, denoted CRS$(h)$, if the map $f : R \rightarrow R(h)$ defined by 
\[ a_i \rightarrow q_i := a_i \imod{h} \] 
is injective or, equivalently, surjective.
\end{mydef}

With the aforementioned definition, we readily glean the following equivalent properties if $R$ is a CRS$(h)$:
\begin{enumerate}
\item[(P1)] if $i \neq j$, then $q_i \neq q_j$; 
\item[(P2)] if $i \neq j$, then $a_i \not \equiv a_j \imod{h}$; and
\item[(P3)] for $a \in \Z$, if $a \equiv a_i \imod{h}$, then $a \not \equiv a_j \imod{h}$ for all $j \neq i$. 
\end{enumerate}

The next two results are well-known number-theoretic results (see, for instance, \cite[Theorem 3.3]{MR1382656}, \cite[Theorem 4.6 and Corollary 4.7]{MR1008474} or \cite[Theorems 54 and 55]{MR2445243}).

% Thm 2------------------------------------------------------------------------------------------------------------------------------------------------
\begin{thm} \label{2} If $\gcd{(k,m)} = d$, then $ka \equiv kb \imod{m}$ if and only if $a \equiv b \imod{(m/d)}$.
\end{thm}

In particular is the following useful case.

% Cor 3-------------------------------------------------------------------------------------------------------------------------------------------------
\begin{cor} \label{3} If $\gcd{(k,m)} = 1$, then $ka \equiv kb \imod{m}$ if and only if $a \equiv b \imod{m}$
\end{cor}

% Thm 4------------------------------------------------------------------------------------------------------------------------------------------------
\begin{thm} \label{4} If $\left\{ a_0,a_1,\dots,a_{h-1} \right\}$ is a complete residue system $\imod{h}$ and $p \in \Z$, then $\left\{ p a_0, p a_1, \dots, p a_{h-1} \right\}$ is a {\it complete residue system} $\imod{h}$ if and only if $\gcd{(p,h)} = 1$.
\end{thm}

% Rem 5------------------------------------------------------------------------------------------------------------------------------------------------
\begin{rem} Quite frequently in the literature (see \cite[Theorem 56]{MR2445243}, \cite[Theorem 3.5]{MR1382656}, \cite[Section 4.1, Exercise 10]{MR1008474}, \cite[Theorem 20(i)]{MR0201369}, or \cite[Theorem 62]{MR0092794}), one finds only {\it sufficiency}, however, as the following proof shows, {\it necessity} holds as well.
\end{rem}

\begin{proof} If $\gcd{(p,h)} = 1$, then, following \hyperref[3]{Corollary 3}, $p a_i \equiv p a_j \imod{h}$ if and only if $a_i \equiv a_j \imod{h}$, which holds if and only if $i =j$.

Conversely, if $\gcd{(p,h)}=d > 1$, then $\left\{ a_0, a_1, \dots, a_{h-1} \right\}$ is not a CRS$(h/d)$. Thus, there exist distinct integers $i$ and $j$ such that $a_i \equiv a_j \imod{(h/d)}$, which holds, by \hyperref[2]{Theorem 2}, if and only if $p a_i \equiv p a_j \imod{h}$, i.e., \\ $\left\{ p a_0, p a_1, \dots, \\ p a_{h-1}\right\}$ is not a CRS$(h)$.
\end{proof}

%--------------------------------------------------------------------------------------------------------------------------------------------------------
\section{The Application}
%--------------------------------------------------------------------------------------------------------------------------------------------------------

For $h > 1$, let $\omega := e^\frac{2 \pi i}{h} \in \C$ and $\Omega_h := \left\{ 1, \omega, \dots, \omega^{h-1}\right\} \subseteq \C$, i.e., $\Omega_h = \left\{ z \in \C: z^h - 1 = 0 \right\}$.

% Thm 6------------------------------------------------------------------------------------------------------------------------------------------------
\begin{thm} \label{6} If $\left\{ a_0,a_1,\dots,a_{h-1} \right\}$ is a set of integers, then \\ $\left\{ \omega^{a_0}, \omega^{a_1}, \dots, \omega^{a_{h-1}} \right\} = \Omega_h$ if and only if $\left\{ a_0,a_1,\dots,a_{h-1} \right\}$ is a complete residue system $\imod{h}$.
\end{thm}

\begin{proof} If $\left\{ a_0, a_1, \dots, a_{h-1} \right\}$ is a CRS$(h)$, then, for each $k \in \{0, 1, \dots, h-1\}$, there exists one (and only one) $q_k \in \{0, 1, \dots, h-1 \}$ such that $a_k \equiv q_k \imod{h}$, i.e., there exists an integer $l_k$ such that $a_k = h l_k + q_k$. For each such $k$, 
\[ \omega^{a_k} = \omega^{h l_k + q_k} = \left( e^\frac{2 \pi i}{h} \right)^{h l_k + q_k} = e^{2 \pi l_k i + \frac{2 \pi q_k i}{h}} = e^{2 \pi l_k i} \cdot e^\frac{2 \pi q_k i}{h} = \left( e^\frac{2 \pi i}{h} \right)^{q_k} = \omega^{q_k}, \]
so that $\left\{ \omega^{a_0}, \omega^{a_1}, \dots, \omega^{a_{h-1}} \right\} = \left\{ \omega^{q_0}, \omega^{q_1}, \dots, \omega^{q_{h-1}} \right\} = \Omega_h$.

Conversely, suppose $\left\{ a_0, a_1, \dots, a_{h-1}\right\}$ is not a CRS$(h)$, and for every $k$, let $q_k := a_k \imod{h}$. Then $\left\{ q_0, q_1, \dots, q_{h-1} \right\} \subset \left\{ 0, 1, \dots, h-1 \right\}$ so that $\left\{ \omega^{a_0}, \omega^{a_1}, \dots, \omega^{a_{h-1}} \right\} \neq \Omega_h$.
\end{proof}

% Cor 7-------------------------------------------------------------------------------------------------------------------------------------------------
\begin{cor} \label{7} For $p > 1$, let $\Omega_h^p := \left\{ 1, \omega^p, \dots, \omega^{(h-1) p} \right\}$. Then $\Omega_h^p = \Omega_h$ if and only if $\gcd{(p,h)} = 1$. 
\end{cor}

% Lem 8------------------------------------------------------------------------------------------------------------------------------------------------
\begin{lem} If $a$, $b$, $c$, and $h \in \Z$, then $a \equiv b \imod{h}$ if and only if $(a+c) \equiv (b + c) \imod{h}$.  
\end{lem}

\begin{proof} Note that 
\begin{align*}
a \equiv b \imod{h} 
&\Leftrightarrow a - b = hk \tag{$k \in \Z$} 			\\
&\Leftrightarrow (a + c) - (b + c) = hk				\\
&\Leftrightarrow (a+c) \equiv (b + c) \imod{h}. \qedhere
\end{align*}
\end{proof}

% Cor 9-------------------------------------------------------------------------------------------------------------------------------------------------
\begin{cor} \label{9} If $\left\{ a_0,a_1,\dots,a_{h-1} \right\}$ is a {\it complete residue system} $\imod{h}$ and $l$ is any integer, then $\left\{ p a_0 + l, p a_1 + l, \dots, p a_{h-1} + l \right\}$ is a {\it complete residue system} $\imod{h}$ if and only if $\gcd{(p,h)}=1$.
\end{cor}

% Lem 10-----------------------------------------------------------------------------------------------------------------------------------------------
\begin{lem} \label{10} If $a$, $b$, $c$ $d$ and $h \in \Z$, then $(a + hc) \equiv (b + hd) \imod{h}$ if and only if $a \equiv b \imod{h}$.
\end{lem}

\begin{proof} Note that 
\begin{align*}
a \equiv b \imod{h} 
&\Leftrightarrow a - b = hk \tag{$k \in \Z$} 			\\
&\Leftrightarrow (a + hc) - (b + hd) = h(k + c - d)		\\
&\Leftrightarrow (a+c) \equiv (b + c) \imod{h}. \qedhere
\end{align*}
\end{proof}

% Cor 11-----------------------------------------------------------------------------------------------------------------------------------------------
\begin{cor} If $\left\{ a_0,a_1,\dots,a_{h-1} \right\}$ and $\left\{ l_0, l_1, \dots, l_{h-1} \right\}$ are sets of integers, then $\left\{ a_0, a_1, \dots, a_{h-1} \right\}$ is a {\it complete residue system} $\imod{h}$ if and only if $\left\{ a_0 + h l_0, a_1 + h l_1, \dots, a_{h-1} + h l_{h-1} \right\}$ is a {\it complete residue system} $\imod{h}$.
\end{cor}

\begin{proof} If $\left\{ a_0 + h l_0, a_1 + h l_1, \dots, a_{h-1} + h l_{h-1} \right\}$ is not a CRS$(h)$, then there exist distinct integers $i$, $j \in \left\{0, 1, \dots, h-1 \right\}$  such that $(a_i+h l_i) \equiv (a_j + h l_j) \imod{h}$. Following \hyperref[10]{Lemma 10}, $(a_i+h l_i) \equiv (a_j + h l_j) \imod{h}$ if and only if $a_i \equiv a_j \imod{h}$; i.e., if and only if $\left\{ a_0, a_1, \dots, a_{h-1} \right\}$ is not a CRS$(h)$.
\end{proof}

% Defn 12----------------------------------------------------------------------------------------------------------------------------------------------
\begin{mydef} For $z = r e^{i \theta} \in \C$, let $z^\frac{1}{p} := r^\frac{1}{p} e^\frac{i \theta}{p}$ and, for $l \in \left\{ 0, 1, \dots, p-1 \right\}$, let $f_l (z) := z^\frac{1}{p} e^\frac{2 \pi l i}{p} = r^\frac{1}{p} e^\frac{(\theta + 2 \pi l) i}{p}$. Thus, $f_l$ denotes the $(l+1)^\tth$ branch of the $p^\tth$-root function.
\end{mydef}

% Thm 13----------------------------------------------------------------------------------------------------------------------------------------------
\begin{thm} \label{13} If $L := \left\{ l = \left( l_0, l_1, \dots, l_{h-1} \right) : l_k \in \left\{ 0, 1, \dots, p-1 \right\} \right\}$, and $(\Omega_h)_l^\frac{1}{p} := \left\{ f_{l_0} ( 1 ), f_{l_1} ( \omega ), \dots, f_{l_{h-1}} ( \omega^{h-1} ) \right\}$, then there exists $l \in L$ such that $(\Omega_h)_l^\frac{1}{p} = \Omega_h$ if and only if $\gcd{(h,p)=1}$.
\end{thm}

% Rem 14----------------------------------------------------------------------------------------------------------------------------------------------
\begin{rem} \label{14} For any $k \in \{ 0,1,\dots,h-1 \}$, note that  
\[ f_{l_k} ( \omega^k ) = \omega^\frac{k}{p} \cdot e^\frac{2\pi l_k i}{p} = e^\frac{2\pi k i}{hp} \cdot e^\frac{2\pi l_k i}{p}	
= e^\frac{2 \pi i \left( k + hl_k \right)}{hp} = \left( e^\frac{2 \pi i}{h} \right)^{\frac{k + h l_k}{p}} = \omega^{\frac{k + h l_k}{p}} \]
so that $(\Omega_h)_l^\frac{1}{p} = \left\{ \omega^{\frac{h l_0}{p}}, \omega^{\frac{1 + h l_1}{p}}, \dots, \omega^{\frac{(h-1) + h l_{h-1}}{p}} \right\}$.
\end{rem}

\begin{proof}
If $\gcd{(p,h)} = 1$, then, following \hyperref[9]{Corollary 9}, for each $k$, the set $\left\{ k + hl \right\}_{l=0}^{p-1}$ is a CRS$(p)$; in particular, there exists $l_k \in\left\{ 0, 1, \dots, p-1 \right\}$ such that $k + h l_k \equiv 0 \imod{p}$, i.e., $k+ h l_k$ is divisible by $p$. Moreover, the set 
\[ \left\{ \frac{k + h l_k}{p} \right\}_{k=0}^{h-1} \] 
is a CRS$(h)$; indeed, following \hyperref[3]{Corollary 3}, 
\[ \left( \frac{i + h l_i}{p} \right) \equiv \left( \frac{j + h l_j}{p} \right) \imod{h} \] 
if and only if $\left( i + h l_i \right) \equiv \left( j + h l_j \right) \imod{h}$, which holds, following \hyperref[10]{Lemma 10}, if and only if $i \equiv j \imod{h}$. Since $\{ 0, 1, \dots, h-1 \}$ is a CRS$(h)$, it follows that $i=j$, and the claim is established.

Conversely, suppose $\gcd{(p,h)} = d > 1$. We claim that for at least one $k$, there does not exist $l_k \in \{ 0, 1, \dots, p-1 \}$ such that $k + h l_k \equiv 0 \imod{p}$; suppose the assertion is false so that for each $k$, there exists $l_k$ such that $k + h l_k \equiv 0 \imod{p}$, i.e., there exists an integer $q_k$ such that $k+ h l_k = p q_k$, or $ h l_k = p q_k - k$ so that $p q_k \equiv k \imod{h}$. Since $k \in \{ 0, 1, \dots, h-1 \}$, the set $\left\{ p q_k \right\}_{k=0}^{h-1}$ is a CRS$(h)$, establishing a contradiction since $\gcd{(p,h)} \neq 1$. Following \hyperref[14]{Remark 14}, at least one exponent in the set $(\Omega_h)_l^\frac{1}{p} := \left\{ \omega^{\frac{k + h l_k}{p}} \right\}_{k=0}^{h-1}$ is not an integer so that $(\Omega_h)_l^\frac{1}{p} \neq \Omega_h$.
\end{proof}

% Cor 15-----------------------------------------------------------------------------------------------------------------------------------------------
\begin{cor} If $ L := \left\{ l = \left( l_0, l_1, \dots, l_{h-1} \right) : l_k \in \left\{ 0, 1, \dots, p-1 \right\} \right\}$, and $(\Omega_h)_l^\frac{q}{p} := \left( \Omega_h^q \right)_l^\frac{1}{p}$ or $(\Omega_h)_l^\frac{q}{p} := \left( \left( \Omega_h \right)_l^\frac{1}{p} \right)^q$, then there exists $l \in L$ such that $(\Omega_h)_l^\frac{q}{p} = \Omega_h$ if and only if $\gcd{(h,p)} = \gcd{(h,q)} = 1$.
\end{cor}

\bibliographystyle{plain}
\bibliography{crs}

% Document End--------------------------------------------------------------------------------------------------------------------------------------
\end{document}